\documentclass[10pt]{article}

\usepackage{amsmath,amsfonts,amssymb,amsthm}
\usepackage{rotating}
\usepackage{blkarray, bigstrut}
\usepackage{epsfig}
\usepackage{enumitem}
\usepackage{xcolor} 
\usepackage{subdepth} 
\usepackage{mathrsfs}
\usepackage{multirow}

\usepackage{caption}
\usepackage{float}
\usepackage{etex}

\usepackage{tikz}
\usetikzlibrary{arrows,matrix}

\usepackage[hidelinks]{hyperref}
\usepackage{amsmath}
\usepackage{eqnarray,amsmath}
\usepackage{amssymb}
\usepackage{amsfonts}
\usepackage[usenames,dvipsnames]{pstricks}
\usepackage{graphicx}
\usepackage{subcaption}
\usepackage[labelformat=parens,labelsep=quad,skip=3pt]{caption}
\usepackage{qtree}
\usepackage{algorithm}
\usepackage[noend]{algpseudocode}
\usepackage[utf8]{inputenc}
\usepackage{amsthm}
\usepackage{pstricks-add,}
\usepackage{mathtools}
\usepackage{dirtytalk}

\usepackage{xcolor}
\usepackage{hyperref}

\usepackage[english]{babel} 
\usepackage[pangram]{blindtext}

\usepackage{array}
\usepackage{rotating}
\usepackage{tikz}
\usepackage{blkarray}

\usepackage{MnSymbol} 

\usepackage{bbm}

\hypersetup{colorlinks=false,linkbordercolor=red,linkcolor=green,pdfborderstyle={/S/U/W 1}}

\newtheorem{theorem}{Theorem}[section]

\newtheorem{proposition}[theorem]{Proposition}

\newtheorem{corollary}[theorem]{Corollary}
\newtheorem{example}[theorem]{Example}

\usepackage{amsmath,amsfonts,amssymb,amsthm}
\usepackage{rotating}
\usepackage{blkarray, bigstrut}
\usepackage{epsfig}
\usepackage{enumitem}
\usepackage{xcolor} 
\usepackage{subdepth} 
\usepackage{mathrsfs}
\usepackage{multirow}

\usepackage{tikz}
\usetikzlibrary{arrows,matrix}

\usepackage{caption}
\usepackage{float}

\usepackage{makecell}
\newcolumntype{V}{!{\vrule width 1pt}}

\usepackage{makecell}
\newcolumntype{V}{!{\vrule width 1pt}}

\begin{document}

\title{A matrix approach to the structure, enumeration, and applications of partially ordered sets\thanks{This work was supported by the National Research Foundation of Korea(NRF) Grant funded by the Korean Government (MSIT)(RS-2025-00573047).}}
\author{Gi-Sang Cheon$^{a}$\thanks{Corresponding author}, Hong Joon Choi$^{b}$, Gukwon Kwon$^{a}$, Hojoon Lee$^{a}$, Yaling Wang$^{c}$\\
{\footnotesize $^a$ \textit{Department of Mathematics, Sungkyunkwan
University (SKKU), Suwon 16419, Rep. of Korea}}\\
{\footnotesize $^{b}$ \textit{Department of Mathematics, Jeju National University, Jeju 63243, Rep. of Korea
}}\\
{\footnotesize $^{c}$ \textit{Department of Mathematics, Shanxi Key Laboratory of Cryptography and Data Security,}}\\ 
{\footnotesize\textit{ Shanxi Normal University, Taiyuan 030031, P.R. China}}\\
{\footnotesize gscheon@skku.edu, hjchoi@jejunu.ac.kr, tomy1995@skku.edu, hojoon1101@g.skku.edu, wang-yaling@hotmail.com
}}

\date{}
\maketitle

\begin{abstract}
We present a matrix-theoretic approach for studying and enumerating finite posets through their incidence representations, referred to as poset matrices.
Naturally labelled posets are encoded as Boolean lower triangular matrices, allowing a unified treatment of Birkhoff’s problem on non-isomorphic posets and Dedekind’s problem on antichains.
A key idea is a systematic construction and indexing of poset matrices as principal submatrices of the binary Pascal matrix, leading to new structural insights through permutation similarity and domination relations. This approach provides a consistent matrix-based perspective on classical enumeration problems in poset theory.
\end{abstract}

\setlength{\baselineskip}{20pt}

\vskip1pc \noindent\textit{AMS classifications}: 05A15, 06A07, 15B34.

\noindent \textit{Key words}: Partial order, poset matrix, binary Pascal matrix, Birkhoff's problem, Dedekind's problem.

\section{Introduction}

A classical enumeration problem posed by Birkhoff \cite{Birkhoff} in 1948  asks for the number of non-isomorphic posets (NIPs) with $n$ elements.
Despite extensive study, this problem remains open, with exact values only known up to $n=16$ \cite{Brinkmann}.
Another fundamental counting problem is Dedekind’s problem \cite{Dedekind}, proposed in 1897, which asks for the number of monotone Boolean functions on $n$ variables, equivalently the number of antichains in the $n$-dimensional Boolean lattice ${\mathbb B}_n$. This problem has also been solved only up to $n=9$ \cite{Jakel}.
These two problems have motivated a wide range of combinatorial and algorithmic approaches in poset theory.

In graph theory, a fruitful interaction between graph properties and matrix representations—such as adjacency, incidence, and Laplacian matrices—has led to powerful algebraic and algorithmic tools.
Although every finite poset can be represented by its Hasse diagram, which is a directed acyclic graph, matrix-based approaches to posets have remained comparatively underdeveloped.
This contrast suggests that a systematic matrix-theoretic approach for posets may provide new structural insights and effective enumeration methods.

The aim of this paper is to develop such a method by studying {\it naturally labelled posets} through their matrix representations, referred to as poset matrices.
A poset can be naturally encoded by a Boolean matrix recording its partial order relations, in a manner analogous to incidence matrices in graph theory.
This perspective allows the application of Boolean matrices to the structural analysis and enumeration of posets.

Throughout this paper, let $X_n=\{0,1,\ldots,n-1\}$ and let ${\mathbb B}=\{0,1\}$ denote the Boolean algebra with operations
\[
a+b=\max\{a,b\}, \qquad a\cdot b=\min\{a,b\},
\quad a,b\in{\mathbb B}.
\]
A $(0,1)$-matrix over ${\mathbb B}$ is called a Boolean matrix.
For a poset $P=(X_n,\preceq)$, the \emph{poset matrix} \cite{Cheon1,Moha} of $P$ is the $n\times n$ Boolean matrix
$A=[a_{ij}]_{i,j\ge 0}$ defined by
\[
a_{ij}=
\begin{cases}
1, & \text{if } j\preceq i,\\
0, & \text{otherwise}.
\end{cases}
\]
The matrix $A$ is reflexive ($a_{ii}=1$ for all $i$), antisymmetric (if $a_{ij}=1$ and $i\neq j$, then $a_{ji}=0$), and transitive (if $a_{ij}=1$ and $a_{jk}=1$, then $a_{ik}=1$).

It is known \cite{kim} that a unit Boolean matrix $A$ is a poset matrix if and only if it is Boolean idempotent, that is, $A^2=A$.
Although posets have been extensively studied from combinatorial and order-theoretic viewpoints, relatively few works investigate them systematically via poset matrices.
In this paper, we show that poset matrices provide a natural and effective tool for the classification and enumeration of finite posets.

The paper is organized as follows.
In Section~2, we study poset matrices of naturally labelled posets and their equivalence under permutation similarity.
In Section~3, we prove that every poset matrix of order $n$ can be realized as an induced principal submatrix of the binary Pascal matrix of order $2^n$.
Section~4 introduces domination relations and Pascal-equivalence classes, and develops their role in the enumeration of non-isomorphic posets.
In Section~5, we investigate duality on induced Pascal submatrices and describe its action on index vectors.
Finally, Section~6 is devoted to an extension of Dedekind’s problem, interpreting the enumeration of antichains of the Pascal poset as solutions $x\in{\mathbb B}^n$ of the Boolean fixed point equation $x{\bf P}_n  = x$ associated with the Pascal matrix ${\bf P}_n$.

\section{Naturally labelled posets and their poset matrices}
 
A partial order $\preceq$ on $X_n$ is said to be {\it natural} if $x \preceq y $ implies $x \le y$. We refer to such posets as {\it naturally labelled} posets (shortly, NL posets) \cite{Avann, Bevan, Dean}. 
The counting sequence for NL posets on $X_n$ is known up to $n=12$
(OEIS A006455 \cite{OEIS}), and their asymptotic enumeration has been studied
in \cite{Bright}.

Notably, every finite partial ordering can be isomorphically represented as a natural partial ordering. This follows from the Szpilrajn extension theorem \cite{Szp}, which asserts that any partial order can be extended to a total order.

\begin{proposition}({\cite{Dean}})\label{Dean}
Every finite poset is isomorphic to a naturally labelled poset.
\end{proposition}

Two posets $P$ and $Q$ are isomorphic, denoted $P\cong Q$, if there exists an order-preserving bijection $\theta: P \rightarrow Q$ whose inverse is order-preserving; that is, $x \preceq y$ in $P$ if and only if $\theta(x) \preceq \theta(y)$ in $Q$.
Proposition~\ref{Dean} plays a fundamental role in the enumeration of finite posets.
It guarantees that every finite poset admits a naturally labelled representative. As a consequence, the problem of counting non-isomorphic posets
can be reduced to the enumeration of naturally labelled posets modulo relabelling.
This observation allows us to work within a fixed ground set $X_n=\{0,\ldots,n-1\}$ and to translate poset enumeration problems into the classification of poset matrices up to permutation similarity.

Let $\mathcal{NL}(n)$ denote the set of naturally labelled (NL) posets on $X_n$.
By definition, if $A$ is the poset matrix of an NL poset $P\in\mathcal{NL}(n)$ then it is an $n\times n$ $(0,1)$-matrix
that is lower triangular with all diagonal entries equal to $1$. In this representation, reflexivity and antisymmetry are automatically established by the unit lower triangular form of $A$.
Consequently, $A$ is the poset matrix of some NL poset on $X_n$
if and only if $A$ is an $n\times n$ Boolean matrix that is unit lower triangular
and transitive.
The transitivity condition admits a simple matrix-theoretic characterization \cite{Bevan}:
$A$ contains no submatrix of the form
\[
\begin{pmatrix}
1 & \mathbf{1} \\
0 & 1
\end{pmatrix},
\]
where the bold entry lies on the main diagonal.

From now on, all poset matrices are assumed as Boolean matrices in unit lower
triangular form associated with naturally labelled posets. There is a natural bijection between NL posets on $X_n$ and their poset matrices.
More precisely, we have the following proposition.

\begin{proposition}[\cite{Bevan}]\label{Bevan}
There is a one-to-one correspondence between NL posets on $X_n$ and
$n\times n$ poset matrices.
\end{proposition}

Let $\mathcal{PM}(n)$ denote the set of $n\times n$ poset matrices.
By Proposition~\ref{Bevan}, we obtain
$$
|\mathcal{NL}(n)| = |\mathcal{PM}(n)|.
$$
In particular, for each NL poset $P\in\mathcal{NL}(n)$ there exists a unique poset matrix
$A\in\mathcal{PM}(n)$, and conversely.
When it is necessary to specify this correspondence explicitly, we write $A_P$
for the poset matrix associated with $P\in\mathcal{NL}(n)$, and $P_A$ for the NL poset
associated with $A\in\mathcal{PM}(n)$.

It should be emphasized that an unlabeled poset may admit several distinct natural
labelings, each of which induces a different poset matrix, see Fig. 1. 

\newcommand{\pmat}[1]{%
{\scriptsize
$\begin{bmatrix}#1\end{bmatrix}$}%
}

\begin{center}
\setlength{\tabcolsep}{4pt} 
\renewcommand{\arraystretch}{0.9}

{\scriptsize\setlength{\arraycolsep}{2pt}

\begin{tabular}{ccccccc}

\begin{minipage}[t]{0.108\textwidth}\centering
\begin{tikzpicture}[scale=0.5,
 dot/.style={circle,fill=black,inner sep=1.3pt},
 lab/.style={font=\scriptsize}]
\node[dot,label={[lab]below:{$0$}}] at (-0.7,0) {};
\node[dot,label={[lab]below:{$1$}}] at ( 0.0,0) {};
\node[dot,label={[lab]below:{$2$}}] at ( 0.7,0) {};
\end{tikzpicture}

\kern5pt 
\pmat{1&0&0\\ 0&1&0\\ 0&0&1}
\end{minipage}
&

\begin{minipage}[t]{0.108\textwidth}\centering
\begin{tikzpicture}[scale=0.5,
 dot/.style={circle,fill=black,inner sep=1.3pt},
 lab/.style={font=\scriptsize}]
\node[dot,label={[lab]left:{$0$}}]  (b0) at (0,0) {};
\node[dot,label={[lab]right:{$1$}}] (b1) at (1,0) {};
\node[dot,label={[lab]right:{$2$}}] (b2) at (0.5,1.1) {};
\draw (b0)--(b2) (b1)--(b2);
\end{tikzpicture}

\kern5pt
\pmat{1&0&0\\ 0&1&0\\ 1&1&1}
\end{minipage}
&

\begin{minipage}[t]{0.108\textwidth}\centering
\begin{tikzpicture}[scale=0.5,
 dot/.style={circle,fill=black,inner sep=1.3pt},
 lab/.style={font=\scriptsize}]
\node[dot,label={[lab]right:{$0$}}] (c0) at (0,0) {};
\node[dot,label={[lab]left:{$1$}}]  (c1) at (-0.45,1.1) {};
\node[dot,label={[lab]right:{$2$}}] (c2) at (0.45,1.1) {};
\draw (c0)--(c1) (c0)--(c2);
\end{tikzpicture}

\kern5pt
\pmat{1&0&0\\ 1&1&0\\ 1&0&1}
\end{minipage}
&

\begin{minipage}[t]{0.108\textwidth}\centering
\begin{tikzpicture}[scale=0.5,
 dot/.style={circle,fill=black,inner sep=1.3pt},
 lab/.style={font=\scriptsize}]
\node[dot,label={[lab]right:{$0$}}] (d0) at (0,0) {};
\node[dot,label={[lab]right:{$1$}}] (d1) at (0,1) {};
\node[dot,label={[lab]right:{$2$}}] (d2) at (0,2) {};
\draw (d0)--(d1)--(d2);
\end{tikzpicture}

\kern5pt
\pmat{1&0&0\\ 1&1&0\\ 1&1&1}
\end{minipage}
&

\begin{minipage}[t]{0.108\textwidth}\centering
\begin{tikzpicture}[scale=0.5,
 dot/.style={circle,fill=black,inner sep=1.3pt},
 lab/.style={font=\scriptsize}]
\node[dot,label={[lab]left:{$0$}}] (e0) at (0,0) {};
\node[dot,label={[lab]left:{$2$}}] (e2) at (0,1) {};
\node[dot,label={[lab]right:{$1$}}] at (1,0) {};
\draw (e0)--(e2);
\end{tikzpicture}

\kern5pt
\pmat{1&0&0\\ 0&1&0\\ 1&0&1}
\end{minipage}
&

\begin{minipage}[t]{0.108\textwidth}\centering
\begin{tikzpicture}[scale=0.5,
 dot/.style={circle,fill=black,inner sep=1.3pt},
 lab/.style={font=\scriptsize}]
\node[dot,label={[lab]left:{$1$}}] (f1) at (0,0) {};
\node[dot,label={[lab]left:{$2$}}] (f2) at (0,1) {};
\node[dot,label={[lab]right:{$0$}}] at (1,0) {};
\draw (f1)--(f2);
\end{tikzpicture}

\kern5pt
\pmat{1&0&0\\ 0&1&0\\ 0&1&1}
\end{minipage}
&

\begin{minipage}[t]{0.108\textwidth}\centering
\begin{tikzpicture}[scale=0.5,
 dot/.style={circle,fill=black,inner sep=1.3pt},
 lab/.style={font=\scriptsize}]
\node[dot,label={[lab]left:{$0$}}] (g0) at (0,0) {};
\node[dot,label={[lab]left:{$1$}}] (g1) at (0,1) {};
\node[dot,label={[lab]right:{$2$}}] at (1,0) {};
\draw (g0)--(g1);
\end{tikzpicture}

\kern5pt
\pmat{1&0&0\\ 1&1&0\\ 0&0&1}
\end{minipage}

\end{tabular}
}

\vskip 1pc
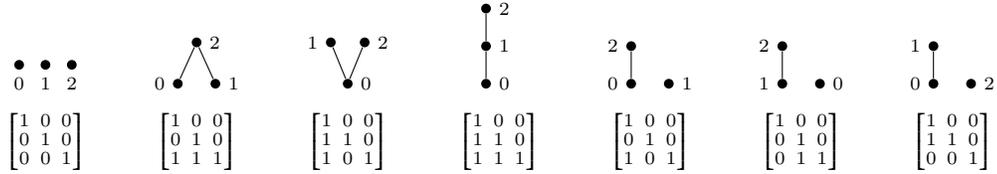
{Fig.~1. The seven NL posets on $X_3$ and their corresponding poset matrices.}
\end{center}

\begin{theorem}
Let $P$ be an unlabeled poset on $X$ with $|X|=n$. If $A_P$ and $A'_P$ are the poset matrices obtained from different natural labelings on $X_n$ of $P$, then there exists an $n\times n$ permutation matrix $Q$ such that $ Q^{T} A_P Q=A'_P$.
\end{theorem}
\begin{proof}
Let $L:X\to X_n$ and $L':X\to X_n$ be different natural labelings of the poset $P$ that produce poset matrices $A_P$ and $A'_P$, respectively. For $x,y\in X$, let $x\preceq y$. By naturally labelling, there exist $j\leq i$ and $v\leq u$ such that $L(x)=j$, $L(y)=i$ and $L'(x)=v$, $L'(y)=u$. Thus $(A_P)_{ij}=1$ and $(A'_P)_{uv}=1$. Define a bijection $\sigma:X_n\to X_n$ by $\sigma(L(x))=L'(x)$ for all $x \in X$, and let $Q_\sigma$ be the $n\times n$ permutation matrix corresponding to $\sigma\in \mathfrak{S}_n$. Using the elementary matrix $E_{i,j}$ with a single $1$ in position $(i,j)$, we have, for all $(i,j)$ with $(A_P)_{ij}=1$,
$$
Q_\sigma^TE_{i,j}Q_\sigma=E_{\sigma(i),\sigma(j)}=E_{u,v}.
$$ 
Since $A=\sum_{j\preceq i} E_{i,j}$ it follows that $Q_\sigma^{T}A_P Q_\sigma=A'_P\in \mathcal{PM}(n)$, which completes the proof.  
\end{proof}

Two poset matrices $A, B \in \mathcal{PM}(n)$ are said to be {\it permutation similar}, written $A\sim B$, if there exits an $n\times n$ permutation matrix $Q$ such that $B=Q^T A Q$.
This relation is an equivalence relation on $\mathcal{PM}(n)$.
We denote by $[A]$ the equivalence class of $A$, and by $\mathcal{PM}(n)/\sim$
the set of all equivalence classes. This equivalence classes correspond precisely to isomorphism classes of NL posets with $n$ elements. 

\begin{theorem}\label{Birk} Let $A,B \in \mathcal{PM}(n)$ be poset matrices associated with the NL posets $P_A$ and $P_B$, respectively.
Then
$$
A \sim B \quad \text{if and only if} \quad P_A \cong P_B .
$$
Moreover, the cardinality $|\mathcal{PM}(n)/{\sim}|$ is equal to the number of
non-isomorphic posets with $n$ elements.
\end{theorem}
\begin{proof} Permutation similarity corresponds precisely to relabeling the elements of a poset.
Thus, two poset matrices are permutation similar if and only if the corresponding
NL posets are isomorphic.
Each equivalence class $[A]$ therefore consists of all poset matrices arising from
the natural labelings of a fixed unlabeled poset. This completes the proof.
\end{proof}

By Theorem~\ref{Birk}, Birkhoff’s question on enumerating the number of
non-isomorphic posets with $n$ elements can be reformulated in matrix-theoretic terms as follows:

\medskip
\noindent
\centerline{``What is the number of non-equivalent $n\times n$ poset matrices over
the Boolean algebra $\mathbb{B}$?"}

\noindent This reformulation naturally leads to the following fundamental question.
\begin{quote}
`` How can one effectively determine or construct non-equivalent
poset matrices?"
\end{quote}
In the subsequent sections, we answer this question by introducing equivalence relations based on permutation similarity and domination relations.

\section{Induced poset matrices of the binary Pascal matrix}

An essential example of a poset matrix in $\mathcal{PM}(n)$ is the
$n\times n$ \emph{binary Pascal matrix}, denoted by ${\bf{P}}_n$.
It is obtained from the Pascal matrix by taking modulo~$2$:
for all $i,j\in\{0,1,\ldots,n-1\}$,
$$
({\bf{P}}_n)_{i,j}\equiv \binom{i}{j}\pmod{2}.
$$
The NL poset associated with ${\bf{P}}_n$, labelled by $0,1,\ldots,n-1$, is called the {\it Pascal poset} on
$X_n$ and is denoted by ${\mathbb{P}}_n$ (see~\cite{Cheon1}).
Each subset of an $n$-element set can be identified with an integer
in $\{0,1,\ldots,2^n-1\}$ via its binary representation.
Consequently, every naturally labelled poset on $X_n$ can be embedded into
the $n$-dimensional \emph{Boolean lattice} $\mathbb{B}_n=(2^{X_n},\subseteq)$, consisting of all subsets of $X_n$ ordered by inclusion
(see also~\cite{Harzheim}).
In particular, the Boolean lattice $\mathbb{B}_n$ is isomorphic to the
Pascal poset $\mathbb{P}_{2^n}$.
Throughout this paper, the term \emph{Pascal matrix} will always mean the
binary Pascal matrix over the Boolean algebra ${\mathbb B}$.

In this section, we show that every poset matrix in $\mathcal{PM}(n)$
arises as an induced principal submatrix of the Pascal matrix
${\bf{P}}_{2^n}$. Denote by $Q_{k,n}$ the collection of all
$k$-element subsets of $X_n=\{0,1,\ldots,n-1\}$. An element $\alpha\in Q_{k,n}$ may be written either as a set $\alpha=\{\alpha_0,\ldots,\alpha_{k-1}\}$ or when an ordering is specified, as a vector 
$\alpha=(\alpha_0,\ldots,\alpha_{k-1})$. For a given matrix $M$, we denote by $M[\alpha]$ or $M[\{\alpha_0,\ldots,\alpha_{k-1}\}]$ the $k\times k$ submatrix of $M$
obtained by taking the rows and columns indexed by the elements of $\alpha$.
In particular, for any $\alpha\in Q_{k,n}$, the matrix
${\bf{P}}_n[\alpha]$ is a $k\times k$ poset matrix. We  refer to ${\bf{P}}_n[\alpha]$ as the
{\it induced principal submatrix} of ${\bf{P}}_n$ associated with the index set $\alpha$.

 {\bf Lucas Theorem.}~For the binary representation of an integer $n\ge1$ of the form $n=\sum_{i\geqslant 0}b_i2^i$ with $b_i\in\{0,1\}$, the {\it support} of $n$ is defined by ${\rm supp}(n):=\{i\mid b_i\ne0\}$. The {\it Lucas theorem} \cite{ELU} asserts that
\begin{eqnarray*}\label{Pas}
\binom{i}{j}\;\equiv 1\pmod2~~\text{if and only if} ~{\rm supp}(j)\subseteq {\rm supp}(i).
\end{eqnarray*}

 
 In the following theorem, we show that every poset matrix can be embedded into the Pascal matrix.
 
\begin{theorem}\label{thm1} Every poset matrix $A\in \mathcal{PM}(n)$ is an induced principal submatrix of the Pascal matrix ${\bf{P}}_{2^n}$.
More precisely, there exists an index set $\alpha=(\alpha_0,\ldots,\alpha_{n-1})\in Q_{n,2^n}$ with $2^{i}\le \alpha_i<2^{i+1}$ such that 
$A = {\bf{P}}_{2^n}[\alpha]$ where the index set $\alpha$ is determined by 
\begin{eqnarray}\label{poseteqn}
A(2^0,\ldots,2^{n-1})^T=(\alpha_0,\ldots,\alpha_{n-1})^T.
\end{eqnarray} 
\end{theorem}
\begin{proof} Let $A=[a_{ij}]\in\mathcal{PM}(n)$. From the equation (\ref{poseteqn}) we obtain 
\begin{eqnarray}\label{bin rep}
\alpha_i=\sum_{j=0}^{n-1}a_{ij}2^{j},\quad i\in X_n.
\end{eqnarray}
Since $a_{ij}\in\{0,1\}$ for $i>j$, $a_{ii}=1$ and $a_{ij}=0$ otherwise, it follows that $2^{i}\le \alpha_i\le 2^{i+1}-1$ for all $i\in X_n$ and $1\le \alpha_0<\cdots<\alpha_{n-1}\le 2^n-1$. Thus $\alpha\in Q_{n,2^n}$. To show that $A={\bf{P}}_{2^n}[\alpha]$, it suffices to show that for all $i,j\in X_n$,
\begin{eqnarray}\label{bij}
a_{ij}\equiv\binom{\alpha_i}{\alpha_j} \equiv 1\pmod2~~\text{equivalently,} ~~{\rm supp}(\alpha_j)\subseteq {\rm supp}(\alpha_i).
\end{eqnarray}
First let $a_{ij}=1$ for $i\ne j$. If $m\in{\rm supp}(\alpha_j)$ then $a_{jm}=1$ from (\ref{bin rep}). By the transitivity of $A$ we have $a_{im}=1$ so that $m\in{\rm supp}(\alpha_i)$. Thus ${\rm supp}(\alpha_j)\subseteq {\rm supp}(\alpha_i)$. By Lucas property, $\binom{\alpha_i}{\alpha_j}\equiv 1 \pmod 2.$ Next, if $a_{ij}=0$ then $j\notin {\rm supp}(\alpha_i)$. Since $j\in {\rm supp}(\alpha_j)$ it follows that ${\rm supp}(\alpha_j)\nsubseteq {\rm supp}(\alpha_i)$. By Lucas theorem,
$\binom{\alpha_i}{\alpha_j}\equiv 0\pmod2.$ Consequently, (\ref{bij}) holds for all $i,j\in X_n$,
which implies that $A$ is the induced principal submatrix of the Pascal matrix ${\bf{P}}_{2^n}$ indexed by $\alpha$, {\it i.e.}, $A={\bf{P}}_{2^n}[\alpha]$.
\end{proof}

A subposet $Q$ is said to be an {\it induced subposet} of a poset ${P}=(X,\preceq)$ if there exists an injection $f:Q\rightarrow P$ such that, for all $u,v\in Q$, $u\preceq v$ if and only if $f(u)\preceq f(v)$.
 
\begin{corollary}\label{thm2} Every poset with $n$ element is isomorphic to an induced NL poset of the Pascal poset ${\mathbb P}_{2^n}$. Moreover, this induced poset is the poset $(\alpha,\preceq)$ where $\alpha$ is determined by the equation (\ref{poseteqn}). We denote it by $\mathbb{P}_{2^n}[\alpha]$. 
  \end{corollary}
  \begin{proof} Let $P$ be a poset with $n$ elements. By Proposition \ref{Dean} we may assume that $P$ is an NL poset on $X_n$. Thus there exists a unique poset matrix $A_P$. By Theorem \ref{thm1}, there exists $\alpha=(\alpha_0,\ldots,\alpha_{n-1})$ with $2^{i}\le \alpha_i<2^{i+1}$ such that $A_P={\bf{P}}_{2^n}[\alpha]$. Moreover, (\ref{bij}) implies that $P$ is isomorphic to $(\alpha,\preceq)$, which is the NL poset induced by the elements $\alpha_1,\ldots,\alpha_{n}$ of the Pascal poset ${\mathbb P}_{2^n}$. Hence the proof is completed. 
  \end{proof}

\begin{example} {\rm 
There are five non-isomorphic posets with three elements.
By Theorem~\ref{thm1} and Corollary~\ref{thm2}, each of these five posets
can be realized as an induced subposet of the Pascal poset
$\mathbb{P}_8$.
More precisely, their associated poset matrices appear as induced
principal submatrices of the Pascal matrix ${\bf{P}}_8$, and the
corresponding induced subposets of $\mathbb{P}_8$ represent NL posets on three elements, as illustrated in Fig.~2.}
\end{example}

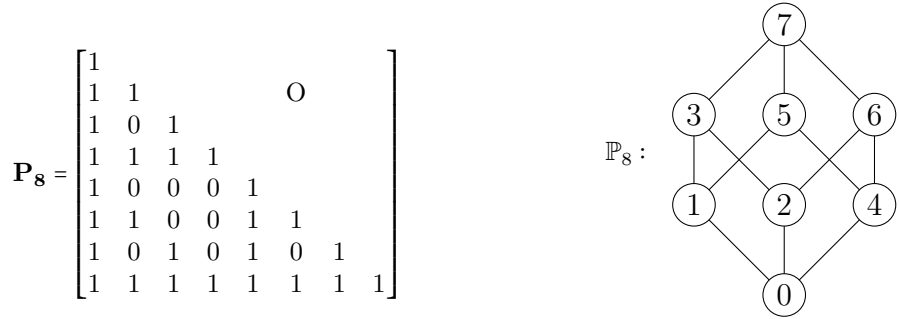
\begin{minipage}{0.4\textwidth}
		\[
		\bf{P}_{8}=
		\begin{bmatrix}
			1&  &  &  &  &  &  &\\
			1& 1&  &  &  &{\rm O } &  &\\
			1& 0& 1&  &  &  &  &\\
			1& 1& 1& 1&  &  &  &\\
			1& 0& 0& 0& 1&  &  &\\
			1& 1& 0& 0& 1& 1&  &\\
			1& 0& 1& 0& 1& 0& 1&\\
			1& 1& 1& 1& 1& 1& 1& 1
		\end{bmatrix}
		\]
	\end{minipage}\quad\quad\quad\quad\quad\quad
	\begin{minipage}{0.4\textwidth}
		${\mathbb P}_8:$\quad\begin{tikzpicture}[scale=0.2, baseline={(current bounding box.center)}]
			\draw (20,0) circle (1.4142135623cm);
			\node at (20,0) {\fontsize{12}{14}\selectfont 0};
			\draw[-] (19,1) to (15,5);
			\draw (14,6) circle (1.4142135623cm);
			\node at (14,6) {\fontsize{12}{14}\selectfont 1};
			\draw[-] (14,7.4142135623) to (14,10.58578644);
			\draw (14,12) circle (1.4142135623cm);
			\node at (14,12) {\fontsize{12}{14}\selectfont 3};
			\draw[-] (15,13) to (19,17);
			\draw (20,18) circle (1.4142135623cm);
			\node at (20,18) {\fontsize{12}{14}\selectfont 7};
			
			\draw[-] (15,7) to (19,11);
			\draw (20,12) circle (1.4142135623cm);
			\node at (20,12) {\fontsize{12}{14}\selectfont 5};
			\draw[-] (20,13.4142135623) to (20,16.58578644);
			
			\draw[-] (20,1.4142135623) to (20,4.58578644);
			\draw (20,6) circle (1.4142135623cm);
			\node at (20,6) {\fontsize{12}{14}\selectfont 2};
			\draw[-] (19,7) to (15,11);
			\draw[-] (21,7) to (25,11);
			
			\draw[-] (21,1) to (25,5);
			\draw (26,6) circle (1.4142135623cm);
			\node at (26,6) {\fontsize{12}{14}\selectfont 4};
			\draw[-] (26,7.4142135623) to (26,10.58578644);
			\draw (26,12) circle (1.4142135623cm);
			\node at (26,12) {\fontsize{12}{14}\selectfont 6};
			\draw[-] (25,13) to (21,17);
			
			\draw[-] (25,7) to (21,11);
			\end{tikzpicture}
	\end{minipage}%
\centerline{Fig. 1: $8\times 8$ Pascal matrix and associated Pascal poset on $[8]$} 
 \vskip.5pc
\noindent For instance, we obtain the induced NL poset $(\{1,3,5\},\prec)$ from the Pascal poset ${\mathbb P}_8$:
     \begin{eqnarray*}
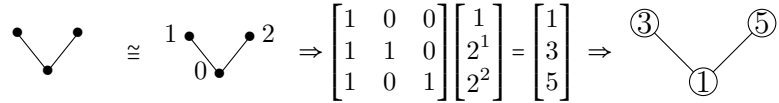
\begin{minipage}[c]{0.12\textwidth}
\centering
	\begin{tikzpicture}[scale=0.1]
	\node at (5,10) {$\bullet$};
	\node at (15.5,10.5) {};
	\draw[-] (5,10) to (9,5);
	\node at (13,10) {$\bullet$};
	\draw[-] (13,10) to (9,5);
	\node at (2.5,10.5) {};
	\node at (9,5) {$\bullet$};
	\node at (6.5,5.5) {};
	\end{tikzpicture} 
		\end{minipage}\cong 
\begin{minipage}[c]{0.12\textwidth}
	\centering
	\begin{tikzpicture}[scale=0.1]
	\node at (5,10) {$\bullet$};
	\node at (15.5,10.5) {2};
	\draw[-] (5,10) to (9,5);
	\node at (13,10) {$\bullet$};
	\draw[-] (13,10) to (9,5);
	\node at (2.5,10.5) {1};
	\node at (9,5) {$\bullet$};
	\node at (6.5,5.5) {0};
	\end{tikzpicture} 
		\end{minipage}
\Rightarrow
\begin{bmatrix}
				1 & 0 & 0 \\
				1 & 1 & 0 \\
				1 & 0 & 1 \\
			\end{bmatrix}\begin{bmatrix}
				1\\
				2^1 \\
				2^2 \\
			\end{bmatrix}=\begin{bmatrix}
				1\\
				3 \\
				5 \\
			\end{bmatrix}~\Rightarrow~
\begin{minipage}[c]{0.12\textwidth}
		\centering
		\begin{tikzpicture}[scale=0.13, baseline={(current bounding box.center)}]
			\draw (20,0) circle (1.4142135623cm);
			\node at (20,0) {\fontsize{12}{14}\selectfont 1};
			\draw[-] (19,1) to (15,5);
			\draw (14,6) circle (1.4142135623cm);
			\node at (14,6) {\fontsize{12}{14}\selectfont 3};
			
			\draw[-] (21,1) to (25,5);
			\draw (26,6) circle (1.4142135623cm);
			\node at (26,6) {\fontsize{12}{14}\selectfont 5};
		\end{tikzpicture}
	\end{minipage}%
\end{eqnarray*}
Similarly, we obtain:
\begin{eqnarray*}
A_1=\begin{bmatrix}
				1 & 0 & 0 \\
				0 & 1 & 0 \\
				0 & 0 & 1 \\
			\end{bmatrix}={\bf{P}_{8}}[\{1,2,4\}]\quad\Leftrightarrow\quad {\mathbb P}_8[\{1,2,4\}]\quad\cong
\begin{minipage}[c]{0.12\textwidth}
	\centering
	\begin{tikzpicture}[scale=0.1]
		\node at (5,5) {$\bullet$};
		\node at (5,8) {};
		\node at (10,5) {$\bullet$};
		\node at (10,8) {};
		\node at (15,5) {$\bullet$};
		\node at (15,8) {};
			\end{tikzpicture}
		\end{minipage}\\
A_2=\begin{bmatrix}
				1 & 0 & 0 \\
				1 & 1 & 0 \\
				0 & 0 & 1 \\
			\end{bmatrix}={\bf{P}_{8}}[\{1,3,4\}]\quad\Leftrightarrow\quad{\mathbb P}_8[\{1,3,4\}]\quad\cong
\begin{minipage}[c]{0.12\textwidth}
	\centering
	\begin{tikzpicture}[scale=0.1]
		\node at (5,5) {$\bullet$};
		\node at (2.5,5.5) {};
		\draw[-] (5,5) to (5,10);
		\node at (5,10) {$\bullet$};
		\node at (2.5,10.5) {};
		\node at (10,5) {$\bullet$};
		\node at (12.5,5.5) {};
			\end{tikzpicture} 
		\end{minipage}\\
A_3=\begin{bmatrix}
				1 & 0 & 0 \\
				1 & 1 & 0 \\
				1 & 0 & 1 \\
			\end{bmatrix}={\bf{P}_{8}}[\{1,3,5\}]\quad\Leftrightarrow\quad {\mathbb P}_8[\{1,3,5\}]\quad\cong 
\begin{minipage}[c]{0.12\textwidth}
	\centering
	\begin{tikzpicture}[scale=0.1]
	\node at (5,10) {$\bullet$};
	\node at (15.5,10.5) {};
	\draw[-] (5,10) to (9,5);
	\node at (13,10) {$\bullet$};
	\draw[-] (13,10) to (9,5);
	\node at (2.5,10.5) {};
	\node at (9,5) {$\bullet$};
	\node at (6.5,5.5) {};
	\end{tikzpicture} 
	\end{minipage}\\
A_4=\begin{bmatrix}
				1 & 0 & 0 \\
				0 & 1 & 0 \\
				1 & 1 & 1 \\
			\end{bmatrix}={\bf{P}}_{8}[\{1,2,7\}]\quad\Leftrightarrow\quad {\mathbb P}_8[\{1,2,7\}]\quad\cong 
\begin{minipage}[c]{0.12\textwidth}
	\centering
	\begin{tikzpicture}[scale=0.1]
		\node at (5,5) {$\bullet$};
		\node at (15.5,5.5) {};
		\draw[-] (5,5) to (9,10);
		\node at (13,5) {$\bullet$};
		\draw[-] (13,5) to (9,10);
		\node at (2.5,5.5) {};
		\node at (9,10) {$\bullet$};
		\node at (6.5,10.5) {};
			\end{tikzpicture}
		\end{minipage}\\
A_5=\begin{bmatrix}
				1 & 0 & 0 \\
				1 & 1 & 0 \\
				1 & 1 & 1 \\
			\end{bmatrix}={\bf{P}_{8}}[\{1,3,7\}]\quad\Leftrightarrow\quad{\mathbb P}_8[\{1,3,7\}]\quad\cong
\begin{minipage}[c]{0.12\textwidth}
	\centering
	\begin{tikzpicture}[scale=0.1]
		\node at (5,5) {$\bullet$};
		\node at (2.5,5.5) {};
		\draw[-] (5,5) to (5,15);
		\node at (5,10) {$\bullet$};
		\node at (2.5,10.5) {};
		\node at (5,15) {$\bullet$};
		\node at (2.5,15.5) {};
			\end{tikzpicture}
		\end{minipage}
\end{eqnarray*}

\centerline{Fig. 2: Non-equivalent poset matrices for $n=3$ and corresponding non-isomorphic posets.} 
\bigskip
Theorem~\ref{thm1} provides a fundamental bridge between the theory of posets and
the combinatorial structure of the Pascal matrix.
By showing that every $n\times n$ poset matrix arises as an induced principal
submatrix of ${\bf P}_{2^n}$, the problem of enumerating non-isomorphic posets
is reduced to the classification of index sets
$\alpha\in Q_{n,2^n}$ up to equivalence relations by permutation similarity.

\section{Domination relations and Pascal-equivalence classes}

In this section, we develop a matrix-theoretic approach to the enumeration of
non-isomorphic posets by introducing the notion of domination relations and Pascal-equivalence classes.
The central idea is to translate the classification of NL posets up to isomorphism into the classification of poset matrices up to
permutation similarity, $A\sim B$.

Let $A,B\in\mathcal{PM}(n)$.
By Theorem~\ref{thm1}, there exist index sets
$\alpha,\beta\in Q_{n,2^n}$ such that $A={\bf{P}}_{2^n}[\alpha]$ and $B={\bf{P}}_{2^n}[\beta]$.
If ${\bf{P}}_{2^n}[\alpha]\sim {\bf{P}}_{2^n}[\beta]$, we say that
$\alpha$ and $\beta$ are {\it Pascal-equivalent}, and write $\alpha\sim_p\beta$.
This defines an equivalence relation on $Q_{n,2^n}$.  The Pascal-equivalence class (or {\it orbit}) of $\alpha\in Q_{n,2^n}$ is $$[\alpha]:=\bigl\{\beta\in Q_{n,2^n}\mid \alpha\sim_p\beta \bigr\},
$$ 
and we denote by $Q_{n,2^n}/\sim_p$ the set of all such equivalence classes. We now define a map
$$
\varphi:\;
\mathcal{PM}(n)/\sim
\;\;\longrightarrow\;\;
Q_{n,2^n}/\sim_p,
\qquad
\varphi([A]) := [\alpha],
$$
where $A={\bf{P}}_{2^n}[\alpha]$. This map is well defined and bijective.
Through the bijection $\varphi$, by Theorem \ref{Birk} Birkhoff's problem of enumerating
NIPs on $X_n$ is reduced to counting the number of Pascal-equivalence classes of index
vectors in $Q_{n,2^n}$.

{\bf Domination relations.} Let $\alpha=(\alpha_0,\ldots,\alpha_{n-1})\in Q_{n,2^n}$.
For each $i=0,\ldots,n-1$, recall that (\ref{bin rep}) is the binary expansion of each $\alpha_i$:
\begin{equation*}\label{binary}
\alpha_i
=
b_{i0}2^0 + b_{i1}2^1 + \cdots + b_{i,n-1}2^{n-1},
\qquad
b_{ij}\in\{0,1\}.
\end{equation*}
The {\it incidence matrix} associated with $\alpha$, denoted by $M_\alpha$, is the
$n\times n$ Boolean matrix defined by $(M_\alpha)_{i,j} := b_{ij}$ for $0\le i,j\le n-1$. Equivalently,
$$
(M_\alpha)_{i,j} =
\begin{cases}
1, & \text{if}\;j\in\operatorname{supp}(\alpha_i),\\
0, & \text{otherwise}.
\end{cases}
$$

Let $R_0,\ldots,R_{n-1}$ denote the row vectors of $M_\alpha$.
By construction, each $R_i$ encodes the support of $\alpha_i$.
A row vector $R_i$ is said to be {\it dominated} by $R_j$ if $R_i \le R_j$, the inequalities being entrywise. Clearly,
\begin{eqnarray*}\label{admissible}
R_i\leq R_j\quad{\text{if and only if}} \quad {\rm supp}(\alpha_i)\subseteq{\rm supp}(\alpha_j).
\end{eqnarray*}
An $n\times n$ Boolean matrix $A$ is said to be {\it domination-equivalent} to $M_\alpha$,  written $A\sim_d M_{\alpha}$, if $A$ and $M_\alpha$ have identical domination relations among their row vectors.

An element $b_{ij}$ of $M_\alpha$ is called {\it changeable} if flipping its value
($0\leftrightarrow 1$) does not create or break any domination relation
between $R_i$ and $R_j$ for all distinct pairs $(i,j)$ of indices.

For instance, let $\alpha=(2,5,9,13)\in Q_{4,16}$.
Then
\begin{eqnarray}\label{Ex2}
M_\alpha =
\begin{pmatrix}
{\color{red}0} & 1 & {\color{red}0} & {\color{red}0} \\
{\color{red}1} & 0 & 1 & 0 \\
{\color{red}1} & 0 & 0 & 1 \\
1 & 0 & 1 & 1
\end{pmatrix}.
\end{eqnarray}
In this case, $R_1\le R_3$ and $R_2\le R_3$.
The five entries highlighted in red are changeable.

\begin{theorem}[Domination operations]\label{domination}
For any $\alpha\in Q_{n,2^n}$, the domination relations among the row vectors of
$M_\alpha$ are preserved under the following operations:
\begin{itemize}
\item [{\rm I.}] Permuting rows or columns of $M_\alpha$;
\item [{\rm II.}] Flipping the changeable elements of $M_\alpha$.
\end{itemize}
\end{theorem}
\begin{proof} Fix $\alpha\in Q_{n,2^n}$ and let $R_0,\ldots,R_{n-1}$ be row vectors of $M_\alpha$. Permuting two rows $R_i$ and $R_j$ merely relabels the indices of the row vectors.
Hence domination relations are unchanged up to relabelling. Similarly, permuting two columns $C_i$ and $C_j$ exchanges the $i$th and $j$th coordinates
of every row vector of $M_\alpha$. Since domination relation is defined componentwise, such a coordinate permutation does not affect any domination relation.
Thus, permuting rows or columns of $M_\alpha$ preserves all domination relations, which proves I. 

By definition, a changeable element of $M_\alpha$ can be flipped without altering
any pairwise comparisons between row vectors. Consequently, flipping all changeable elements preserves the domination relations, which proves II.
\end{proof}

If $A$ is obtained from $M_\alpha$ by applying a finite sequence of domination operations, then $A$ is domination-equivalent to $M_\alpha$.
Thus, the domination operations play a fundamental role in identifying and generating the 
Pascal-equivalence class of a given index vector $\alpha\in Q_{n,2^n}$.

\begin{theorem}\label{main} Let $\alpha=(\alpha_0,\ldots,\alpha_{n-1})\in Q_{n,2^n}$ be fixed and assume that $A$ is domination-equivalent to $M_\alpha$.
If $A{\bf v}_n^{T}=\beta^{T}$ where ${\bf v}_n:=(2^0,\ldots,2^{n-1})$, then $\alpha$ and $\beta$ are Pascal-equivalent, equivalently $\beta\in[\alpha]$.
\end{theorem}
\begin{proof} Assume that $A\sim_d M_{\alpha}$ and that $A{\bf v}_n^{T}=\beta^{T}.$ By the definition of the incidence matrix of an index vector, the equality
$A{\bf v}_n^{T}=\beta^{T}$ implies that $A=M_{\beta}$. Since $M_\beta$ is the incidence matrix obtained from $M_\alpha$ by a finite sequence of domination operations, 
the vector $\beta=(\beta_0,\ldots,\beta_{n-1})$ is not necessarily ordered increasingly. However, by Theorem~\ref{domination}, all domination relations among the
row vectors are preserved throughout this process. Consequently, the domination relations among the row vectors of $M_\beta$ coincide with those of $M_\alpha$.
Therefore, the induced poset matrices satisfy
$$
{\bf{P}}_{2^n}[\alpha] \sim {\bf{P}}_{2^n}[\beta],
$$
which implies that $\alpha$ and $\beta$ are Pascal-equivalent. Hence $\alpha\sim_p \beta$, completing the proof.
\end{proof}

{\bf Algorithmic interpretation.} Theorem~\ref{main} shows that the domination operations admit an algorithmic interpretation for
computing the Pascal-equivalence class  $[\alpha]$ of a given index vector
$\alpha\in Q_{n,2^n}$. Starting from the incidence matrix $M_\alpha$, one applies all
possible domination operations of type~I and type~II to generate new domination-equivalent matrices $A=M_\beta$ satisfying $M_\alpha\sim_d A$.
For each such matrix $A$, the vector
$$
\beta^T := A\,(2^0,\ldots,2^{n-1})^T
$$
is computed and interpreted as an element of $Q_{n,2^n}$, viewed as an
unordered subset of $\{0,\ldots,2^n-1\}$.
By Theorem~\ref{main}, every vector $\beta$ obtained in
this way satisfies $\beta\sim_p\alpha$.
Hence, the domination operations define a constructive procedure for
enumerating all index vectors in the Pascal--equivalence class $[\alpha]$.
\medskip

For instance, let $M_\alpha$ with $\alpha=(2,5,9,13)\in Q_{4,16}$ be the same matrix in (\ref{Ex2}).
By applying the domination operation {\rm I} to $M_\alpha$, and then the domination
operation {\rm II} to the resulting matrix, Theorem~\ref{domination} yields
\begin{eqnarray*}
M_\alpha =
\begin{pmatrix}
{\red 0} & 1 & 0 & 0\\
{\red 1} & 0 & 1 & 0 \\
{\red 1} & 0 & 0 & 1 \\
1 & 0 & 1 & 1
\end{pmatrix}\;\sim_d\; M_\beta =\begin{pmatrix}
1 & 0 & 0 & {\red0}\\
0 & 1 & 0 & {\red1} \\
0 & 0 & 1 & {\red1} \\
0 & 1 & 1 & 1
\end{pmatrix}\;\sim_d\; M_\gamma =\begin{pmatrix}
1 & 0 & 0 & {\red0}\\
0 & 1 & 0 & {\red0} \\
0 & 0 & 1 & {\red1} \\
0 & 1 & 1 & 1
\end{pmatrix}\;\sim_d\;M_\omega=
\begin{pmatrix}
1 & 0 & 0 & {\red0}\\
0 & 1 & 0 & {\red0} \\
0 & 0 & 1 & {\red0} \\
0 & 1 & 1 & 1
\end{pmatrix},
\end{eqnarray*}
where $\alpha=(2,5,9,13)$, $\beta=(1,10,12,14)$, $\gamma=(1,2,12,14)$, $\omega=(1,2,4,14)$. Consequently, by Theorem~\ref{main}, we have
$$
{\bf{P}}_{2^3}[\alpha]={\bf{P}}_{2^3}[\beta]={\bf{P}}_{2^3}[\gamma]={\bf{P}}_{2^3}[\omega]=\begin{pmatrix}
1 & 0 & 0 & 0\\
0& 1 & 0 & 0 \\
0 & 0 & 1 & 0 \\
0 & 1 & 1 & 1
\end{pmatrix}.
$$

\begin{theorem} Let $\alpha\in Q_{n,2^n}$. Then $M_\alpha$ is domination equivalent to a poset matrix.
That is, there exists a poset matrix $A\in\mathcal{PM}(n)$ such that $M_\alpha \sim_d A$.
\end{theorem}

\begin{proof} Let $\alpha=(\alpha_0,\ldots,\alpha_{n-1})$ and let $R_i(M_\alpha)$ denote the $i$th row vectors of $M_\alpha$. By domination operation I, we may assume that $\alpha_0<\alpha_1<\cdots<\alpha_{n-1}$.
Define a Boolean matrix $A=[a_{ij}]$ by
$$
a_{ij}=
\begin{cases}
1, & \text{if } R_j(M_\alpha)\le R_i(M_\alpha),\\
0, & \text{otherwise}.
\end{cases}
$$
We claim that $A$ is the poset matrix which is domination equivalent to $M_\alpha$. Let $R_i=:R_i(M_\alpha)$. Clearly $A$ satisfies $a_{ii}=1$ for all $i$, and $a_{ij}=0$ for all $i,j$ with $i<j$. Thus, $A$ is a unit lower triangular matrix.
Moreover, if $a_{ij}=1$ and $a_{jk}=1$, then $R_j\le R_i$ and $R_k\le R_j$, hence $R_k\le R_i$, implying $a_{ik}=1$. Thus $A$ is transitive. Hence $A\in\mathcal{PM}(n)$ is a poset matrix.

To show $M_\alpha \sim_d A$, we claim that for all $i,j\in X_n$, $R_j(M_\alpha)\le R_i(M_\alpha)$ if and only if $R_j(A)\le R_i(A)$. Indeed, if $R_j(M_\alpha)\le R_i(M_\alpha)$ then $a_{ij}=1$. By the transitivity of $A$, if $a_{jk}=1$ for any $k<j$ then $a_{ik}=1$, which implies $R_j(A)\le R_i(A)$. Conversely, if $R_j(M_\alpha)\not\le R_i(M_\alpha)$ then $a_{ij}=0$. Since $a_{jj}=1$, we have $R_j(A)\not\le R_i(A)$. Therefore $M_\alpha$ and $A$ have identical domination relations among their row vectors. This implies $M_\alpha \sim_d A$, which completes the proof.
\end{proof}

\begin{theorem} All the entries below the main diagonal of a poset matrix $A\in{\cal PM}(n)$ are not changeable.
   \end{theorem}
\begin{proof}
    Let $A=[a_{ij}]$. Suppose that there is a changeable element in $A$, say $a_{ij}$ with $i>j$. If $a_{ij}=1$ then changing $a_{ij}$ by $1\rightarrow 0$ breaks domination relation $R_j\le R_i$ from $a_{jj}=1$. Let $a_{ij}=0$. If $a_{jk}=1$ and $a_{ik}=0$ for some $k<j$ then changing $a_{ij}$ by $0\rightarrow 1$ breaks transitivity of $A$, and changing $a_{ij}$ by $0\rightarrow 1$ breaks domination relationships of $A$ otherwise. Thus $a_{ij}$ can not be changeable for all $i,j$ with $i>j$.
\end{proof}

\begin{proposition}
Let $\alpha=(\alpha_0,\dots,\alpha_{n-1})\in Q_{n,2^n}$. Then ${\bf{P}}_{2^n}[\alpha]=I_n$ if and only if $\{\alpha_0,\dots,\alpha_{n-1}\}$ is an antichain in the Boolean lattice ${\mathbb B}_n$.
\end{proposition}
\begin{proof} If ${\bf{P}}_{2^n}[\alpha]=I_n$, then $\bigl({\bf{P}}_{2^n}[\alpha]\bigr)_{i,j}=0$ for $i\neq j$. For brevity, let ${\rm S}(\alpha_i):={\rm supp}(\alpha_i)$ be the support of $\alpha_i$. By Lucas Theorem, ${\rm S}(\alpha_j)\nsubseteq {\rm S}(\alpha_i)$ and ${\rm S}(\alpha_i)\nsubseteq {\rm S}(\alpha_j)$. Hence no two distinct elements of $\alpha_0,\dots,\alpha_{n-1}$ are comparable, and $\alpha$ is an antichain.

Conversely, if $\{\alpha_0,\dots,\alpha_{n-1}\}$ is an antichain, then
${\rm S}(\alpha_j)\nsubseteq {\rm S}(\alpha_i)$ and ${\rm S}(\alpha_i)\nsubseteq {\rm S}(\alpha_j)$ for $i\neq j$, so all off-diagonal entries
of ${\bf{P}}_{2^n}[\alpha]$ are $0$, while diagonal entries are $1$.
Thus ${\bf{P}}_{2^n}[\alpha]=I_n$.
\end{proof}

\begin{corollary}\label{antichain}
Let $\alpha=(\alpha_0,\dots,\alpha_{n-1})\in Q_{n,2^n}$. If $|{\rm S}(\alpha_0)|=\cdots=|{\rm S}(\alpha_{n-1})|$, equivalently all row sums of $M_\alpha$ are identical then ${\bf{P}}_{2^n}[\alpha]=I_n$.
\end{corollary}

For instance, let $\alpha=(7,11,13,14)\in Q_{4,2^4}$. Then
$$
M_{\alpha}=\begin{bmatrix}
    1 & 1 & 1 & 0 \\
    1 & 1 & 0 & 1 \\
    1 & 0 & 1 & 1 \\
    0 & 1 & 1 & 1
\end{bmatrix},
$$
so that ${\rm S}(7)=\{0,1,2\}$, ${\rm S}(11)=\{0,1,3\}$, ${\rm S}(13)=\{0,2,3\}$, ${\rm S}(14)=\{1,2,3\}$. Since all these subsets have the same cardinality, no one can be subset of another. By Corollary \ref{antichain}, we see that $P_{2^4}[\alpha]=I_4$. 

For brevity, let $c\alpha=(c\alpha_0,\ldots,c\alpha_{k-1})$ and $\alpha+c=(\alpha_0+c,\ldots,\alpha_{k-1}+c)$ for a nonzero real number $c$.

\begin{theorem}\label{even-odd}
Let $\alpha=(\alpha_0,\ldots,\alpha_{n-1})\in Q_{n,2^n}$.
\begin{itemize}
\item[{\rm(a)}] If all components $\alpha_0,\ldots,\alpha_{n-1}$ are even, then $\frac{\alpha}{2}\in[\alpha]$ and $\alpha+1\in[\alpha]$.
\item[{\rm(b)}] If all components $\alpha_0,\ldots,\alpha_{n-1}$ are odd, then $\frac{\alpha-1}{2}\in[\alpha]$ and $\alpha-1\in[\alpha]$.
\end{itemize}
\end{theorem}
\begin{proof}
(a) Assume that all entries of $\alpha$ are even.
Let $M_\alpha$ be the incidence matrix associated with $\alpha$, and denote its column vectors by $C_0,C_1,\ldots,C_{n-1}$. First, permute the columns of $M_\alpha$ cyclically as
$(C_1,\ldots,C_{n-1},C_0)$. The resulting matrix is precisely the incidence matrix corresponding to the index vector $\alpha/2$.
By Theorem~\ref{main}, $\alpha/2\in[\alpha]$. Next, note that all entries in the column $C_0$ are 0 and these are changeable.
Flipping these zero entries to $1$, from position $n-1$ down to position $0$,
yields the incidence matrix corresponding to the index vector $\alpha+1$.
Again by Theorem \ref{main}, we conclude that $\alpha+1\in[\alpha]$. This proves (a).

(b) Assume that all components of $\alpha$ are odd. Then all entries in the column $C_0$ of $M_\alpha$ are $1$, and these entries are
changeable. Flipping them to $0$, from position $0$ up to position $n-1$, yields the incidence matrix corresponding to $\alpha-1$,
whose entries are all even. Applying part~(a) to the index vector $\alpha-1$ completes the proof of~(b).
\end{proof}

\section{Duality on induced Pascal submatrices}

Every poset $P$ admits a \emph{dual poset}, usually denoted by $P^{*}$, which is defined on the same
underlying set with the order relation reversed.
That is, for elements $x,y\in P$, $x \preceq y$ in $P$ if and only if $y \preceq x$ in $P^*$.
Consequently, any statement that holds for all posets also holds for their duals.

Let $A=[a_{ij}]$ be the poset matrix of an NL poset $P$ on
$X_n=\{0,1,\ldots,n-1\}$, and let $A^{*}$ denote the poset matrix of the dual poset $P^{*}$.
Since $P^{*}$ is obtained from $P$ by relabeling each element $i\in X_n$ as $n-1-i$,
it follows (see also \cite{Cheon5}) that $A^{*}$ is obtained from $A$ by reflecting
all entries across the anti-diagonal.
This operation is known as the \emph{flip transpose} of $A$ and is denoted by $A^{F}$.
Explicitly, the entries of $A^{*}=A^{F}$ are given by
$$
(A^{*})_{i,j} \;=\; a_{\,n-1-j,\;n-1-i},
\qquad 0 \le i,j \le n-1.
$$
Equivalently, if $E_n$ denotes the $n\times n$ anti-diagonal permutation matrix, then 
$$
A^{*} = A^{F} = E_n A^{T} E_n.
$$

 For instance,
\begin{eqnarray*}
&&P:\;\quad\begin{minipage}[c]{.20\textwidth}
\begin{tikzpicture}
  [scale=.7,auto=center,every node/.style={circle,fill=blue!20}]

\node[circle] (c1) at (6.2,2.5) {1};
  \node[circle] (c2) at (7,1.5) {0};
  \node[circle] (c4) at (7.8,2.5) {2};
  \node[circle] (c8) at (7.8,3.7)  {3};

      \draw (c1) -- (c2);
       \draw (c4) -- (c2);
   \draw (c8) -- (c4);
 \end{tikzpicture}
\end{minipage}\Leftrightarrow\quad
A=\begin{bmatrix}
 1 & 0 & 0&0\\
 1 & 1 & 0&0\\
 1 & 0 & 1&0\\
 1 & 0 & 1&1
 \end{bmatrix}\\
&& P^*:\quad\begin{minipage}[c]{.20\textwidth}
\begin{tikzpicture}
  [scale=.7,auto=center,every node/.style={circle,fill=blue!20}]
\node[circle] (c1) at (6.2,3) {2};
  \node[circle] (c2) at (7.8,1.7) {0};
  \node[circle] (c4) at (7.8,3) {1};
  \node[circle] (c8) at (7,4)  {3};
      \draw (c4) -- (c2);
    \draw (c8) -- (c4);
   \draw (c8) -- (c1);
 \end{tikzpicture}
\end{minipage}\Leftrightarrow\quad
A^*=\begin{bmatrix}
 1 & 0 & 0&0\\
 1 & 1 & 0&0\\
 0 & 0 & 1&0\\
 1 & 1 & 1&1
 \end{bmatrix}=A^F
\end{eqnarray*}
\vskip.3pc
\begin{center}
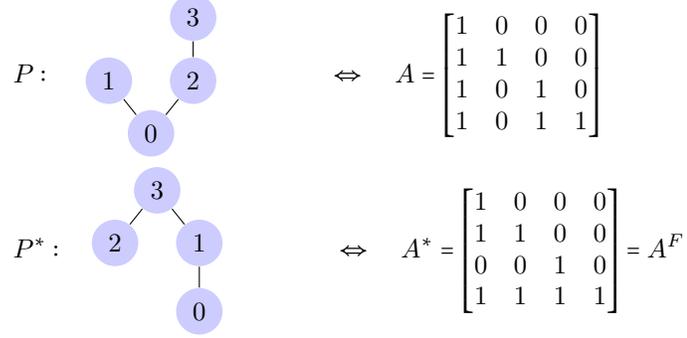
{\rm Fig.~3. Dual poset $P^*$ on $X_4$ and it dual poset matrix $A^*$.}
\end{center}
\bigskip

Let $\alpha_i\in\{0,\ldots,2^n-1\}$ be written in binary form $\alpha_i=\sum_{j=0}^{n-1} b_{ij}2^j$, $b_{ij}\in\{0,1\}$,  in (\ref{bin rep}). The \emph{bitwise complement} of $\alpha_i$ is defined by
\begin{equation*}
\bar{\alpha}_i
:=\sum_{j=0}^{n-1} (1-b_{ij})2^j.
\end{equation*}
Then, for each $i=0,\ldots,n-1$, we have $\bar{\alpha}_i=(2^n-1)-\alpha_i$. Recall that 
$$
{\rm S}(\alpha_i):={\rm supp}(\alpha_i)=\{j\in\{0,\ldots,n-1\}\mid b_{ij}\ne0\}. 
$$

\begin{theorem}\label{main2}
Let $\alpha=(\alpha_0,\ldots,\alpha_{n-1})\in Q_{n,2^n}$ and let
$A={\bf{P}}_{2^n}[\alpha]$ be the induced poset matrix of the binary Pascal matrix.
Then the dual matrix $A^*$ is the poset matrix given by $A^*={\bf{P}}_{2^n}[\beta]$,
where
$$
\beta=(\bar{\alpha}_{n-1},\ldots,\bar{\alpha}_0)\in Q_{n,2^n},
\qquad
\bar{\alpha}_i=(2^n-1)-\alpha_i,
\quad 0\le i\le n-1.
$$
\end{theorem}

\begin{proof}
Let $A={\bf{P}}_{2^n}[\alpha]$ for some
$\alpha=(\alpha_0,\ldots,\alpha_{n-1})\in Q_{n,2^n}$.
Since $A^*$ is the poset matrix in ${\cal PM}(n)$, by Theorem \ref{thm1} there exists a $\beta=(\beta_0,\ldots,\beta_{n-1})\in Q_{n,2^n}$ such that $A^*={\bf{P}}_{2^n}[\beta]$. 
We first establish the following symmetry of the binary Pascal matrix:
for any $\alpha_i,\alpha_j\in\{0,1,\ldots,2^n-1\}$,
\begin{equation}\label{pascal-complement}
({\bf{P}}_{2^n})_{\alpha_i,\alpha_j}
=
({\bf{P}}_{2^n})_{\bar{\alpha}_j,\bar{\alpha}_i},
\end{equation}
where $\bar{\alpha}_k=(2^n-1)-\alpha_k$.
Indeed, by Lucas theorem
$({\bf{P}}_{2^n})_{\alpha_i,\alpha_j}=1$ if and only if
${\rm S}(\alpha_j)\subseteq {\rm S}(\alpha_i)$.
This condition is equivalent to
$$
{\rm S}(2^n-1-\alpha_i)\subseteq {\rm S}(2^n-1-\alpha_j),
$$
which in turn implies
$({\bf{P}}_{2^n})_{\bar{\alpha}_j,\bar{\alpha}_i}=1$,
establishing (\ref{pascal-complement}).
Now, for the induced submatrix $A={\bf{P}}_{2^n}[\alpha]$, we have
$A_{i,j}=({\bf{P}}_{2^n})_{\alpha_i,\alpha_j}$.
Since $(A^*)_{i,j}=A_{\,n-1-j,\,n-1-i}$, it follows from (\ref{pascal-complement}) that
$$
(A^*)_{i,j}
=({\bf{P}}_{2^n})_{\alpha_{n-1-j},\alpha_{n-1-i}}
=({\bf{P}}_{2^n})_{\bar{\alpha}_{n-1-i},\bar{\alpha}_{n-1-j}}.
$$
Define $\beta_i:=\bar{\alpha}_{n-1-i}$ for $0\le i\le n-1$. Then $(A^*)_{ij}=({\bf{P}}_{2^n})_{\beta_i,\beta_j}$, and hence
$A^*={\bf{P}}_{2^n}[\beta]$ with
$\beta=(\bar{\alpha}_{n-1},\ldots,\bar{\alpha}_0)$.
Since $\bar{\alpha}_i=(2^n-1)-\alpha_i$, we clearly have
$\beta\in Q_{n,2^n}$, completing the proof.
\end{proof}

A poset $P$ is said to be {\it self-dual} if $P$ and $P^*$ are isomorphic. Thus, if $A_P=A_P^*$ then $P$ is self-dual.
\begin{corollary} Let $A={\bf{P}}_{2^n}[\alpha]$ be a poset matrix for some $\alpha\in Q_{n,2^n}$. Then $A=A^*$
if and only if $\alpha_i+\alpha_{n-1-i}=2^n-1$ for all $i=0,1,\ldots,n-1$. Consequently, if $n$ is odd, then there is no $\alpha\in Q_{n,2^n}$ with $A=A^*$. 
If $n$ is even, then such $\alpha$ exists; one may freely choose
$\alpha_0<\cdots<\alpha_{n/2-1}$ and then $\alpha_{n-1-i}=(2^n-1)-\alpha_i$.
\end{corollary}

For example, let $\alpha=(0,1,3,12)\in Q_{4,16}$. Then, by Theorem~\ref{main2},
$$
\beta=(\bar{\alpha}_{3},\bar{\alpha}_{2},\bar{\alpha}_{1},\bar{\alpha}_{0})
=(15-12,\;15-3,\;15-1,\;15-0)
=(3,12,14,15).
$$
Hence, 
$$
{\bf{P}}_{16}[\alpha]=\begin{bmatrix}
 1 & 0 & 0&0\\
 1 & 1 & 0&0\\
 1 & 1 & 1&0\\
 1 & 0 & 0&1
 \end{bmatrix}\quad{\rm and}\quad {\bf{P}}_{16}[\alpha]^*={\bf{P}}_{16}[\beta]=\begin{bmatrix}
 1 & 0 & 0&0\\
 0 & 1 & 0&0\\
 0 & 1 & 1&0\\
 1 & 1 & 1&1
 \end{bmatrix}.
 $$
If $\alpha=(0,5,10,15)\in Q_{4,16}$, then
$$
\beta=(15-15,\;15-10,\;15-5,\;15-0)
=(0,5,10,15)=\alpha.
$$
Therefore, ${\bf{P}}_{16}[\alpha]$ coincides with its dual, and the associated NL poset is self-dual:
$$
{\bf{P}}_{16}[\alpha]=\begin{bmatrix}
 1 & 0 & 0&0\\
 1 & 1 & 0&0\\
 1 & 0 & 1&0\\
 1 & 1 & 1&1
 \end{bmatrix}={\bf{P}}_{16}[\alpha]^*.
$$

\begin{theorem}
Let $\alpha^*=(\bar{\alpha}_{n-1},\ldots,\bar{\alpha}_0)$ where $\bar{\alpha}_i=(2^n-1)-\alpha_i$. Then ${\bf{P}}_{2^n}[\alpha]\sim {\bf{P}}_{2^n}[\beta]$ if and only if ${\bf{P}}_{2^n}[\alpha^*]\sim {\bf{P}}_{2^n}[\beta^*]$.
\end{theorem}

\begin{proof}
Let $A={\bf{P}}_{2^n}[\alpha]$ and $B={\bf{P}}_{2^n}[\beta]$ for some
$\alpha,\beta\in Q_{n,2^n}$.
Assume that $A\sim B$. Then there exists a permutation matrix $Q$ such that
$B=Q^{T}AQ$.
Taking the dual (flip transpose) of $B$, we obtain
$$
B^*
=E_n B^{T} E_n
=E_n(Q^{T}A^{T}Q)E_n
=(E_n Q^{T} E_n)\,(E_n A^{T} E_n)\,(E_n Q E_n).
$$
Since $E_n Q E_n$ is a permutation matrix and $E_n A^{T} E_n = A^*$, we have $A^*\sim B^*$. Thus,
${\bf{P}}_{2^n}[\alpha^*]\sim {\bf{P}}_{2^n}[\beta^*]$, completing the proof.
\end{proof}

       \section{An application to the Dedekind's problem}
       
       A monotone Boolean function of $n$ variables is a map $f:\{0,1\}^n \longrightarrow \{0,1\}$ such that if $x\le y$ in coordinatewise then $f(x)\le f(y)$. 
The {\it Dedekind number}, denoted by $M(n)$, is the number of monotone Boolean functions of $n$ variables.

 {\bf Dedekind's problem} (Dedekind, 1897 \cite{Dedekind}) What is the Dedekind number $M(n)$ for $n\ge0$?

\noindent This problem remains computationally difficult, and
exact values are known only for small $n\le 9$ (\cite{Jakel}, OEIS A014466).
It is well known \cite{Hess} that $M(n)$ admits several equivalent combinatorial
interpretations.
In particular, $M(n)$ is equal to the number of antichains of the
$n$-dimensional Boolean lattice $\mathbb B_n$, or equivalently, the number of
order ideals of $\mathbb B_n$.

In this section, we present a new matrix-theoretic formulation
of the Dedekind's problem. Using the Pascal poset $\mathbb P_n$ and its associated poset matrix
${\bf P}_n$ which is the Pascal matrix, we establish an equivalence between order ideals of $\mathbb P_n$ and Boolean vectors $x\in\mathbb B^n$ satisfying the {\it Boolean fixed point equation} $x{\bf P}_n  = x$. 
This equivalence provides a direct correspondence between the combinatorial
structures arising in Dedekind's problem and the solutions of the Boolean fixed point equation. From this viewpoint, the present work opens the door to
matrix-theoretic and algorithmic approaches to the Dedekind's problem. 

We begin by considering the natural question: what is the number of antichains of the Pascal poset $\mathbb P_n$ for $n\ge0$? We call this number the $n$th {\it Dedekind-Pascal number} and denote it by $D_P(n)$. For $n=0$, the Pascal poset is empty, and hence it admits only the empty antichain. Therefore, $D_P(0)=1$. Clearly, for all $k\ge0$, $D_P(2^k)=M(k)$ where $M(k)$ denotes the $k$th Dedekind number. In this sense, the numbers $D_P(n)$ may be viewed as a natural generalization of the classical Dedekind's numbers.

We illustrate the Dedekind-Pascal numbers $D_P(n)$ for $0\le n\le 9$ in the table below. The values highlighted in bold correspond to the Dedekind numbers
$M(k)$ for $n=2^k$, $k=0,1,2,3$:
$$
\begin{array}{c|cccccccccc}
n &0 & {\bf 1} & {\bf2} & 3 & {\bf4} & 5 & 6 & 7 & {\bf8} & 9\\\hline
D_P(n)&1 & {\bf 2} & {\bf 3} & 5 & {\bf 6} & 11 & 14 & 19 & {\bf 20} & 39 
\end{array}
$$
We note that this sequence is also known as the number of antichains in the first $n$ elements of the {\it infinite} Boolean lattice and has been computed up
to $n=50$ (see OEIS, A132581). While the sequence itself is known, the present work introduces a matrix-theoretic approach that connects antichain enumeration with Boolean
fixed-point equations, thereby offering new insight into the Dedekind problem.

Let $P=(X,\preceq)$ be a finite poset. Recall \cite{Stanley}  that an {\it order ideal} (or down-set) of $P$ is a subset $J$ of $P$, such that if $x\in J$ and $y\preceq x$ imply $y\in J$. For $x\in P$, the {\it principal} order ideal of $P$ generated by $x$ is a {\it down set} of the form $\downarrow x:=\{y\in P\mid y \preceq x\}$. Let $\mathfrak{J}(P)$ denote the set of order ideals of $P$. Then $\mathfrak{J}(P)$ forms a distributive lattice ordered by set inclusion where the join is union and the meet is intersection. It is referred to as the {\it ideal lattice} of $P$. It is also well known \cite{Stanley} that there is a one-to-one correspondence between order ideals and antichains of a finite poset. Consequently, counting antichains of the Pascal poset ${\mathbb P}_n$ is equal to enumerating order ideals of ${\mathbb P}_n$. 

We now examine order ideals of the Pascal poset $\mathbb P_n$ from a matrix-theoretic viewpoint. Let $R_0,\ldots,R_{n-1}$ be the row vectors of the Pascal matrix $\mathbf P_n$. By the definition of order ideal, the support of each row vector $R_i$ coincides with the principal order ideal $\downarrow i$ of $\mathbb P_n$. That is,  
$$
{\rm S}(R_i)=\{j\in X_n\mid (R_i)_j=1\}=\downarrow i,\quad i=0,\ldots,n-1.
$$ 
In addition, for an integer $t\in\{0,\ldots,n-1\}$ with binary expansion $t=\sum_{i\geqslant 0}b_i2^i$ with $b_i\in\{0,1\}$, we define the support of $t$ by ${\rm S}(t)=\{i\mid b_i\ne0\}$. 

\begin{proposition}\label{pas} Let~${\mathbb S}_n=\bigl\{{\rm S}(0),\ldots,{\rm S}(n-1)\bigr\}$ be the collection of all supports of $0,\ldots,n-1$. Assume $k$ is minimal with $n\le 2^k$ for an integer $n\ge1$. Then $({\mathbb S}_n,\subseteq)$ is an induced subposet of the Boolean lattice ${\mathbb B}_k$. Moreover, ${\mathbb S}_n$ is isomorphic to the Pascal poset ${\mathbb P}_n$. 
\end{proposition}
\begin{proof}
Let $k$ be minimal such that $n\le 2^k$.
Since every integer $t\in\{0,\ldots,n-1\}$ admits a binary expansion
\[
t=\sum_{i=0}^{k-1} b_i 2^i,\qquad b_i\in\{0,1\},
\]
its support ${\rm S}(t)=\{\,i \mid b_i=1\,\}$ is a subset of $\{0,\ldots,k-1\}$.
Hence $({\mathbb S}_n,\subseteq)$ is an induced subposet of the Boolean lattice
${\mathbb B}_k=(2^{[k]},\subseteq)$ where $2^{[k]}$ denotes the collection of all $2^k$ subsets of the $k$-set. Define a map $\varphi:\{0,\ldots,n-1\}\to{\mathbb S}_n$ by $\varphi(t)={\rm S}(t)$.
By Lucas' theorem, for $x,y\in\{0,\ldots,n-1\}$,
\[
\binom{y}{x}\equiv 1 \pmod{2}
\quad\Longleftrightarrow\quad
{\rm S}(x)\subseteq {\rm S}(y).
\]
Consequently, $\varphi$ is an order-preserving bijection from the Pascal poset
${\mathbb P}_n$ onto $({\mathbb S}_n,\subseteq)$, and thus
${\mathbb S}_n$ is isomorphic to the Pascal poset ${\mathbb P}_n$.
\end{proof}

\begin{example}\label{ex3} {\rm Let $n=5$. For each $t\in\{0,1,2,3,4\}$ with binary expansion, we have
$$
{\mathbb S}_5=\bigl\{\emptyset,\ \{0\},\ \{1\},\ \{0,1\},\ \{2\}\bigr\}
$$
which is a poset ordered by inclusion $\subseteq$, see Fig. 4. By identifying $t$ $(0\le t\le4)$ with ${\rm S}(t)$, i.e.,
\begin{eqnarray*}
0\mapsto{\rm S}(0)=\emptyset,\;1\mapsto{\rm S}(1)=\{0\},\;2\mapsto{\rm S}(2)=\{1\},\;3\mapsto{\rm S}(3)=\{0,1\},\;4\mapsto{\rm S}(4)=\{2\},
\end{eqnarray*}
we obtain the Pascal poset ${\mathbb P}_5$ and corresponding poset matrix is the Pascal matrix ${\bf P}_5$:}
\begin{center}
\begin{tikzpicture}[x=0.6cm,y=0.6cm,>=stealth,
  dot/.style={circle,fill=black,inner sep=1.6pt}
]

{\rm
\begin{scope}
  \node[dot,label=below:{$0\mapsto\emptyset$}] (0) at (0,0) {};
  \node[dot,label=left:{$1\mapsto\{0\}$}]  (1) at (-1.2,1.5) {};
  \node[dot,label=right:{$2\mapsto\{1\}$}] (2) at (1.2,1.5) {};
  \node[dot,label=above:{$3\mapsto\{0,1\}$}] (3) at (0,3) {};
  \node[dot,label=right:{$4\mapsto\{2\}$}] (4) at (2.4,0.8) {};

  \draw (0) -- (1);
  \draw (0) -- (2);
  \draw (0) -- (4);
  \draw (1) -- (3);
  \draw (2) -- (3);

\end{scope}
}

\begin{scope}[xshift=7.5cm]
\node (matrix) at (0,1.6) {$
{\bf P}_5
=\begin{bmatrix}
1 & 0 & 0 & 0 & 0\\[2pt]
1 & 1 & 0 & 0 & 0\\[2pt]
1 & 0 & 1 & 0 & 0\\[2pt]
1 & 1 & 1 & 1 & 0\\[2pt]
1 & 0 & 0 & 0 & 1
\end{bmatrix}
$};
\end{scope}

\end{tikzpicture}
\vskip.3pc
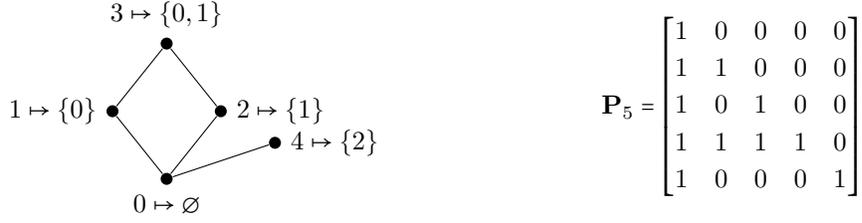
{\rm Fig.~4. The Pascal poset ${\mathbb P}_5\;(\cong\;{\mathbb S}_5)$ and its Poset matrix.}
\end{center}
\end{example}

Denote by $\mathfrak J({\mathbb P}_n)$ the set of all order ideals of ${\mathbb P}_n$, ordered by inclusion, and by $\mathcal A({\mathbb P}_n)$ the set of all
antichains of ${\mathbb P}_n$. For $A\in\mathcal A({\mathbb P}_n)$, define
$$
\downarrow A:=\{x\in X_n \mid x\preceq a \text{ for all } a\in A\}.
$$
Then the set $\downarrow A$ is downward closed by definition and therefore an order ideal of ${\mathbb P}_n$.

For example, consider the Pascal poset ${\mathbb P}_5$ in Fig. 4. If $A=\{2,4\}\in\mathcal A({\mathbb P}_5)$ then $\downarrow A=\{0,2,4\}$:

\begin{center}
\begin{tikzpicture}[
  scale=0.7,
  node distance=10mm and 12mm,
  every node/.style={inner sep=1pt, minimum size=6mm},
  level distance=1.2cm,
  >=stealth
]

\node (0) at (0,0) {\(\varnothing\)};

\node (a) at (-3,1.8) {\(\{0\}\)};
\node (b) at (-1.5,1.8) {\(\{1\}\)};
\node (c) at (0,1.8) {\(\{2\}\)};
\node (d) at (1.5,1.8) {\(\{3\}\)};
\node (e) at (3,1.8) {\(\{4\}\)};

\node (bc) at (-2,3.6) {\(\{1,2\}\)};
\node (be) at (0,3.6) {\(\{1,4\}\)};
\node (ce) at (2,3.6) {\(\{2,4\}\)};
\node (de) at (4,3.6) {\(\{3,4\}\)};

\node (bce) at (0,5.2) {\(\{1,2,4\}\)};

\draw (0) -- (a);
\draw (0) -- (b);
\draw (0) -- (c);
\draw (0) -- (d);
\draw (0) -- (e);

\draw (b) -- (bc);
\draw (c) -- (bc);

\draw (b) -- (be);
\draw (e) -- (be);

\draw (c) -- (ce);
\draw (e) -- (ce);

\draw (d) -- (de);
\draw (e) -- (de);

\draw (bc) -- (bce);
\draw (be) -- (bce);
\draw (ce) -- (bce);

\end{tikzpicture}

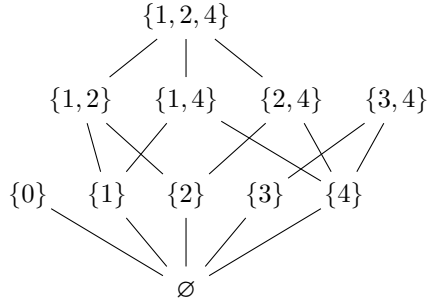
{\text{Fig. 5. $\mathcal{A}({\mathbb P}_5)\cong$ Ideal lattice $\mathfrak{J}({\mathbb P}_5)$}} with $D_P(5)=11$. 
\end{center}
\medskip

A $k$-element antichain ${\cal A}:=\{i_1,\ldots,i_k\}\subseteq X_n$ of the Pascal poset $\mathbb P_n$ corresponds precisely to the induced $k\times k$ identity submatrix of the Pascal matrix $\mathbf P_n$, namely
$$
{\mathbf P}_n[i_1,\ldots,i_k]=I_k.
$$
Equivalently, for any $k\ge 0$ and any $\alpha\in Q_{k,n}$, the condition ${\bf P}_n[\alpha]=I_k$ holds if and only if $\alpha$ corresponds to a $k$-element antichain of the Pascal poset $\mathbb{P}_n$. Consequently, the Dedekind-Pascal number $D_P(n)$ counts all such index sets
$\alpha\in Q_{k,n}$ taken over all $0\le k\le n$.  The following corollary provides a matrix-theoretic interpretation of the
Dedekind-Pascal number $D_P(n)$.

\begin{corollary}
The Dedekind-Pascal number $D_P(n)$ is equal to the total number of principal
induced submatrices of the Pascal matrix ${\bf P}_n$ that are identity
matrices, of all possible sizes $0\le k\le n$, where $I_0$ corresponds to the empty set.
\end{corollary}

We end this section by considering the {\it Boolean fixed-point equation} $x{\bf P}_n = x$ $(x\in \mathbb{B}^n)$, where matrix multiplication is taken over the Boolean algebra.
\begin{theorem} For $x=(x_0,\ldots,x_{n-1})\in\mathbb{B}^n$, let ${\rm S}(x):=\{ i\in X_n\mid x_i=1\}$ be the support of $x\in{\mathbb B}^n$. Then $x{\bf P}_n=x$ if and only if ${\rm S}(x)$ is an order ideal of the Pascal poset ${\mathbb P}_n$. Moreover,
\begin{eqnarray}\label{fixed}
D_P(n)=\bigl|\{x\in\mathbb{B}^n : x{\bf P}_n =x\}\bigr|.
\end{eqnarray}
\end{theorem}
\begin{proof} Let $x=(x_0,\ldots,x_{n-1})\in\mathbb{B}^n$. We first rewrite the fixed-point equation in coordinates.  Recall that ${\bf P}_n$ is the poset matrix of the
Pascal poset ${\mathbb P}_n$, that is, $({\bf P}_n)_{i,j}=1$ if and only if $j \preceq i$ in ${\mathbb P}_n$. Since for each $j\in X_n$,
$$
(x{\bf P}_n)_j=\sum_{i=0}^{n-1}x_i({\bf P}_n)_{i,j}=\sum_{i:j\preceq i} x_i,
$$
it follows that the Boolean fixed point equation $x{\bf P}_n=x$ is equivalent to the coordinate-wise conditions
\begin{eqnarray}\label{fixed-1}
x_j =\sum_{i:j\preceq i} x_i\;\;{\text{for all $j\in X_n$}}.
\end{eqnarray}
Assume that $x{\bf P}_n=x$. Let $i\in {\rm S}(x)$, so $x_i=1$.  
If $j\preceq i$, from (\ref{fixed-1}) we have $x_j=1$, which shows that $j\in {\rm S}(x)$.  
Hence ${\rm S}(x)$ is an order ideal of ${\mathbb P}_n$.

Conversely, assume that ${\rm S}(x)$ is an order ideal of ${\mathbb P}_n$.
Fix $j\in X_n$ and we show that (\ref{fixed-1}) holds. First let $x_j=1$. Since $j\preceq j$ for all $j\in X_n$ it follows that $\sum_{i:j\preceq i} x_i=1$ under Boolean sum, 
so the equality holds in (\ref{fixed-1}).
Now let $x_j=0$. Suppose, to the contrary, that
$\sum_{i:j\preceq i} x_i=1$.
Then there exists some $i$ such that $j\preceq i$ and $x_i=1$, i.e., $i\in {\rm S}(x)$.  
Since ${\rm S}(x)$ is an order ideal and $j\preceq i$, this implies $j\in {\rm S}(x)$, so $x_j=1$, a contradiction.
In both cases (\ref{fixed-1}) holds for every $j$, so $x{\bf P}_n=x$.

Consequently, the set of solutions of the equation $x{\bf P}_n=x$ is in bijection with the set of order ideals of the Pascal poset ${\mathbb P}_n$.
Thus, by the well-known correspondence between order ideals and
antichains, we have
$$
\bigl|\{x\in\mathbb{B}^n : x{\bf P}_n =x\}\bigr|
=|\mathfrak{J}({\mathbb P}_n)|
=\#\{\text{antichains of }{\mathbb P}_n\}=D_P(n),
$$
which proves (\ref{fixed}). Hence the proof is complete.
\end{proof}

For instance, consider the lattice $\mathcal{A}({\mathbb P}_5)$ in  Fig. 5. The Boolean fixed point equation $x{\bf P}_5 = x$ has exactly $11$ solutions.
These solutions are in bijection with the order ideals, or equivalently the antichains, of the Pascal poset $\mathbb P_5$ as shown in Table 6.
\vskip1pc

\begin{center}
\renewcommand{\arraystretch}{1.15}
\begin{tabular}{c|c|c}
\hline
Antichain $\mathcal{A}$ 
& Order ideal $\downarrow\mathcal{A}$ 
& Fixed point $x\in{\mathbb B}^5$ \\ 
\hline
$\varnothing$ 
& $\varnothing$ 
& $(0,0,0,0,0)$ \\[3pt]

$\{0\}$ 
& $\{0\}$ 
& $(1,0,0,0,0)$ \\[3pt]

$\{1\}$ 
& $\{0,1\}$ 
& $(1,1,0,0,0)$ \\[3pt]

$\{2\}$ 
& $\{0,2\}$ 
& $(1,0,1,0,0)$ \\[3pt]

$\{3\}$ 
& $\{0,1,2,3\}$ 
& $(1,1,1,1,0)$ \\[3pt]

$\{4\}$ 
& $\{0,4\}$ 
& $(1,0,0,0,1)$ \\[3pt]

$\{1,2\}$ 
& $\{0,1,2\}$ 
& $(1,1,1,0,0)$ \\[3pt]

$\{1,4\}$ 
& $\{0,1,4\}$ 
& $(1,1,0,0,1)$ \\[3pt]

$\{2,4\}$ 
& $\{0,2,4\}$ 
& $(1,0,1,0,1)$ \\[3pt]

$\{3,4\}$ 
& $\{0,1,2,3,4\}$ 
& $(1,1,1,1,1)$ \\[3pt]

$\{1,2,4\}$ 
& $\{0,1,2,4\}$ 
& $(1,1,1,0,1)$ \\[3pt]

\hline
\end{tabular}
\vskip1pc
Table 6. Antichains of ${\mathbb P}_5$ and coreesponding order ideals and fixed points.
\end{center}

\vskip1pc
We end this section by noting a structural distinction of the Pascal poset
$\mathbb P_n$.
When $n=2^k$, the poset $\mathbb P_n$ is isomorphic to the Boolean lattice
$\mathbb B_k$ and hence $\mathbb P_{2^k}$ forms the Boolean lattice.
For $n\neq 2^k$, the poset $\mathbb P_n$ is generally not a lattice, since
joins or meets may fail to exist.
Nevertheless, our matrix-theoretic formulation via Boolean fixed-point
equations remains valid in this non-lattice setting, providing a natural
extension of the classical Dedekind problem beyond the Boolean lattice case.

\end{document}